\documentclass[11pt, reqno]{amsproc}
\usepackage{amssymb}
\usepackage{amsfonts}
\usepackage{amsmath}
\usepackage{hyperref}
\RequirePackage[english]{babel}
 \usepackage[active]{srcltx}

\theoremstyle{plain}
\newtheorem{corollary}{Corollary}
\newtheorem{lemma}{Lemma}
\newtheorem{theorem}{Theorem}
\theoremstyle{remark}
\newtheorem*{remark}{Remark}

 \newtheorem{theoA}{Theorem}

\numberwithin{equation}{section}
 \let\le\leqslant 
 \let\ge\geqslant

\newcommand{\T}{\mathbb{{T}}}

 \newcommand{\ds}{\displaystyle}

 \newcommand{\Opp}{\mathcal{D}}
 \newcommand{\Oi}{\mathcal{I}}
 \newcommand{\PP}{\Pi}

 \newcommand{\w}[2]{w^{\left(#1,#2\right)}}
 
\newcommand{\LL}[2]{L_{#1}^{(#2)}}
 \newcommand{\Lw}[2]{#1,(#2)}
 \newcommand{\Ll}[1]{#1}
 \newcommand{\Ls}[2]{L_{#1}#2}

\newcommand{\cm}[1]{\nu_{#1}}

\begin{document}

\title[Jacobi weights, fractional integration,
 and sharp Ulyanov inequalities 
]
{
Jacobi weights, fractional integration,\\ 
and sharp Ulyanov inequalities 
}
\author{
   Polina~Glazyrina}
\address{P.\,Yu.~Glazyrina,
Subdepartment of Mathematical Analysis and Theory of
Functions, Ural Federal University, pr. Lenina 51, 620083 Ekaterinburg, Russia}
\email{polina.glazyrina@usu.ru}

\author{Sergey~Tikhonov}
\address{S.~Tikhonov, ICREA and Centre de Recerca Matem\`{a}tica,
Apartat 50~08193 Bellaterra, Barce\-lona, Spain}
\email{stikhonov@crm.cat}

\thanks{
The first  author was supported by the Ural Federal University development program
with the financial support of young scientists.
The second  author was partially supported by
 MTM 2011-27637, 2009 SGR 1303, RFFI 13-01-00043, and NSH-979.2012.1.
 }
\date{May 19, 2013}

\begin{abstract}
We consider functions $L^p$-integrable with Jacobi weights on $[-1,1]$
  and
 prove Hardy--Littlewood type inequalities for fractional integrals.
 As applications, we obtain the sharp $(L_p, L_q)$ Ulyanov-type inequalities for the Ditzian--Totik moduli of smoothness and the $K$-functionals of fractional order.
\end{abstract}
\maketitle

\section{Introduction}
The following $(L_p,L_q)$ inequalities of Ulyanov-type 
 between moduli of smoothness of functions on $\T$
play an important role in approximation theory and functional analysis
(see, e.g., \cite{devore, Di-Ti-2005, har}):


  \begin{equation}\label{ul}
   \omega^r\left(f, t\right)_q \le C
    \left(
    \int_0^t \left(u^{-{\sigma}}\omega^{r}(f,u)_p \right)^{q_1} \frac{du}{u}
    \right)^{1/q_1},
   \end{equation}
   where $r\in \mathbb{N}$, $0< p\le q\le \infty$, ${\sigma}=\frac1p-\frac1q$,  and
   $q_1=\begin{cases}
        q, & q<\infty\\
        1, & q=\infty
        \end{cases}.
   $
Here the $r$-th moduli of smoothness of a function $f\in L_p(\T)$  is given by
\begin{equation*}
\omega^{{r}}\left(f,\delta\right)_{p} = \sup_{|h|\le\delta}
\left\| \Delta^r_h f (x) \right\|_{L^p(\T)},\quad 1\le p \le \infty,
\end{equation*}
where
$$
\Delta^{{r}}_h f (x) = \Delta^{{r}-1}_h\left(\Delta_h f (x)\right) \quad\text{and}
\quad \Delta_h f (x)=f(x+h)-f(x).
$$
Recently (\cite{sharp, trebels}) the sharp version of (\ref{ul}) was proved  in the case  $1<p<q<\infty$: 
  \begin{equation}\label{sharp}
   \omega^r\left(f, t\right)_q \le C
    \left(
    \int_0^t \left(u^{-{\sigma}}\omega^{r+{\sigma}}(f,u)_p \right)^{q_1} \frac{du}{u}
    \right)^{1/q},
   \end{equation}
where $\omega^{r}(f,u)_p$ is the moduli of smoothness 
   of the (fractional) order $r>0.$
Moreover, it turned out  that (\ref{sharp}) also holds if $(p,q)=(1,\infty)$;
see \cite{sharp1}.
In this case $\sigma=1$ and  one can work with the classical (not necessary fractional) moduli of smoothness.
On the other hand, (\ref{sharp}) is not true (\cite{sharp1}) for $1=p<q<\infty$ or $1<p<q=\infty$.

 In the present paper, we consider a nonperiodic case, namely  $L_p$ spaces with Jacobi weights on an interval,
 and  obtain inequalities similar to (\ref{sharp})  for the fractional $K$-functionals and Ditzian--Totik moduli of smoothness.
 We start with notation.

\medskip
  Denote by $ \w{a}{b}(x)=(1-x)^a(1+x)^b$,  $a,b>-1$,  the  Jacobi weight on $[-1,1]$.
   For $1\le p<\infty$, let $\LL{p}{a,b}$ be  the space of all functions $f$  measurable on $[-1,1]$
   with the finite norm
 $$
  \|f\|_{\Lw{p}{a,b}}=\left(\int_{-1}^1 |f(x)|^p \w{a}{b}(x)dx \right)^{1/p}.
 $$
 If   $a=b=0$,  we write  $L_p=\LL{p}{a,b}$,
 $\| \cdot \|_{\Ll{p}}=\| \cdot \|_{\Lw{p}{0,0}}$.
 In the case $p=\infty$,  we set $\LL{p}{a,b}:=C[-1,1]$ and
 $$
   \|f\|_{\Lw{\infty}{a,b}}= \|f\|_{\Ll{\infty}}=\max_{x\in [-1,1]}|f(x)|.
 $$
For an arbitrary interval $[x_1,x_2]$, we set
 $$ \|f\|_{\Ls{p}{[x_1,x_2]}}=\left(\int_{x_1}^{x_2} |f(x)|^p dx \right)^{1/p}, \ 1\le p<\infty,
   \quad  \|f\|_{\Ls{\infty}{[x_1,x_2]}}=\max_{x\in [x_1,x_2]}|f(x)|.
$$

   For $\alpha,\, \beta>-1$, denote by $\psi^{(\alpha,\beta)}_k(x)$, $k=0,\,1,\,\ldots$,  the system of
   Jacobi polynomials orthogonal on $[-1,1]$ with the weight
  $ \w{\alpha}{\beta}$
    and normalized by the condition
     $$
     \int_{-1}^1 \left|\psi^{(\alpha,\beta)}_k(x)\right|^2\w{\alpha}{\beta}(x)dx=1.
     $$
  The Jacobi polynomials
  are the eigenfunctions of the differential operator
  $$
 \Opp=\Opp_2^{(\alpha,\beta)}=\frac{-1}{\w{\alpha}{\beta}(x)}\frac{d}{dx}\w{\alpha}{\beta}(x)
   (1-x^2) \frac{d}{dx},
  $$
  $$
   \Opp \psi^{(\alpha,\beta)}_k=\left(\lambda_k^{(\alpha,\beta)}\right)^2\psi^{(\alpha,\beta)}_k, \qquad
  \lambda_k^{(\alpha,\beta)}=\left(k(k+\alpha+\beta+1)\right)^{1/2}.
  $$
For a function $f \in \LL{p}{\alpha,\beta}$,  $1\le p\le \infty$,  the Fourier--Jacobi expansion is defined as follows:
   \begin{equation}\label{falphabeta1}
    f(x)\sim \sum_{k=0}^\infty
   \widehat{f}^{(\alpha,\beta)}_k  \psi_k^{(\alpha,\beta)}(x),
   \end{equation}
    where
   \begin{equation*}
  \widehat{f}^{(\alpha,\beta)}_k=
   \int_{-1}^1 f(x) \psi_k^{(\alpha,\beta)}(x)\w{\alpha}{\beta}(x) dx, \quad k=0,\,1,\,2,\,\ldots
   \end{equation*}
   Let $\sigma >0$.
  If there exists a function $g\in \LL{1}{\alpha,\beta}$ such that its Fourier--Jacobi expansion has the form
   \begin{equation*}
     g\sim  \sum_{k=1}^\infty
  \left(\lambda_k^{(\alpha,\beta)}\right)^{\sigma}
   \widehat{f}^{(\alpha,\beta)}_k  \psi_k^{(\alpha,\beta)},
   \end{equation*}
  then we use the notation
 $$
 g= \Opp_\sigma^{(\alpha,\beta)} f
 $$
 and we call  $\Opp_\sigma^{(\alpha,\beta)} f $
   the fractional derivative of order $\sigma$ of the function $f$.
  If there exists a function $h\in \LL{1}{\alpha,\beta}$ such that its Fourier--Jacobi expansion has the form
  $$
   h \sim \widehat{f}^{(\alpha,\beta)}_0+ \sum_{k=1}^\infty
  \left(\lambda_k^{(\alpha,\beta)}\right)^{-\sigma}
   \widehat{f}^{(\alpha,\beta)}_k  \psi_k^{(\alpha,\beta)},
  $$
  then we use the notation
 $$
 h= \Oi_\sigma^{(\alpha,\beta)} f
 $$
 and we call $\Oi_\sigma^{(\alpha,\beta)} f$
   the fractional integral of order $\sigma$ of the function $f$.
Notice that  $\Oi_\sigma^{(\alpha,\beta)}$, $\sigma>0$, is a bounded linear operator
on $ \LL{1}{\alpha,\beta}$  (see, e.g.,  \cite[Sec. 5, pp. 789--790]{Ba-1972}).

The $K$-functional corresponding to the differential operator~$\Opp^{(\alpha,\beta)} $ and a real positive number $r$
 is defined by
  \begin{align}\label{K-func}
   &K^r(f, \Opp^{(\alpha,\beta)}_r,t)_{\Lw{p}{\alpha,\beta}}=
  \inf \Big\{\|f-g\|_{\Lw{p}{\alpha,\beta}}+t^{r}\|\Opp^{(\alpha,\beta)}_{r} g\|_{\Lw{p}{\alpha,\beta}}: \
g\in W^{r, (\alpha,\beta)}_{p, (\alpha,\beta)}  \Big\}
  \end{align}
(see  \cite[(1.9)]{Di-1998}),
 where
 $W^{r, (\alpha,\beta)}_{p, (\alpha,\beta)}
 =\Big\{g: g,\,\Opp^{(\alpha,\beta)}_{r} g\in \LL{p}{\alpha,\beta}\Big\}.$

\medskip
\medskip
The main result of this paper is the following 
 \begin{theorem} \label{thU1short}
{\it
  Let   $1< p<q < \infty,$ $r>0,$   $\alpha \ge \beta >-1,$ $\alpha \ge -1/2.$
  Suppose also that
 $$\sigma=(2\alpha+2)\left(\frac1p-\frac1q\right).$$
If~$f\in \LL{p}{\alpha,\beta}$ and
 $$
 \int_0^1 \left(u^{-\sigma} K^{r+\sigma}(f, \Opp^{(\alpha,\beta)}_{r+\sigma},u)_{\Lw{p}{\alpha,\beta}} \right)^{q} \frac{du}{u}<\infty,
 $$
 then $f\in \LL{q}{\alpha,\beta}$
and
  \begin{equation*}\label{thU1_01short}
  K^r(f, \Opp^{(\alpha,\beta)}_r,t)_{\Lw{q}{\alpha,\beta}}\le C
     \left(
    \int_0^t \left(u^{-\sigma} K^{r+\sigma}(f, \Opp^{(\alpha,\beta)}_{r+\sigma},u)_{\Lw{p}{\alpha,\beta}} \right)^{q}
    \frac{du}{u}
    \right)^{1/q}.
   \end{equation*}
  }
\end{theorem}

\medskip
The rest of the paper is organized as follows.
 In Section 2 we obtain the key result  to get sharp Ulyanov inequalities --  the weighted inequalities of
 Hardy--Littlewood and Landau type for functions defined on the interval $[-1,1]$.
 Section 3 contains the definition of fractional $K$-functionals with Jacobi weights and sharp Ulyanov inequalities for $K$-functionals
 (Theorem~\ref{thU1}).
  In Section 4 analogous results for the Ditzian--Totik moduli of smoothness are obtained.
  Namely, we study a relationship between these moduli and the corresponding $K$-functionals
  and prove sharp Ulyanov inequalities for the Ditzian--Totik moduli in the case of $1\le p\le q\le \infty$ (Theorem~\ref{thU}).

\medskip

 \section{Inequalities for fractional integrals with Jacobi weights}
 \subsection{Landau-type inequalities}

We will need the following Hardy-type inequality (see, e.g., \cite{br} and \cite[Theorem 6.2, Example 6.8]{Op-Ku-90}).
 We set $\frac 1q:=0$ for $q=\infty$.
  \begin{theoA}\label{cor1R}
  {\it Let $1 \le p \le q \le\infty,$ $(p,q)\neq (\infty,\infty),$ $a>-\frac1q,$
 $\overline{x}\in (0,\infty)$. Then the inequality
  \begin{equation*}
     \left\|f(x) x^a\right\|_{\Ls{q}{[0,\overline{x}]}}\le C(p,q,a,\overline{x})\left\|f'(x) x^{a+h}\right\|_{\Ls{p}{[0,\overline{x}]}}
  \end{equation*}
  holds for any locally absolutely continuous function $f$  on $(0,\overline{x}]$
  with the property
${f(\overline{x})=0}$ if and only if
 $h \le 1-\left(\frac1p-\frac 1q\right).$
}
    \end{theoA}
\noindent  Let us mention that the quantity $C(p,q,a,\overline{x})$
is nondecreasing with respect to $\overline{x}$.

 The following Landau--type inequality  can be found in, e.g.,
  \cite[Ch. 2, Th. 5.6, p.~38]{DeVoreLorentz1993}.
    \begin{theoA}\label{TL}
{\it  For $1\le p\le \infty,$  $\ell\ge 2,$  there is a constant $C(\ell)$
  such that for  all ${r=0,\ldots,\ell}$ and any function  $f$ with
 $f^{(\ell-1)}$ absolutely continuous
on $\left[-\frac12,\frac12\right]$ and
${f^{(\ell)}\in\Ls{p}{\left[-\frac12,\frac12\right]}}$ we have
  \begin{equation*}
  \left\|f^{(r)}\right\|_{\Ls{p}{\left[-\frac12,\frac12\right]}}\le C(\ell)
   \left(\|f\|_{\Ls{p}{\left[-\frac12,\frac12\right]}}+ \left\|f^{(\ell)}\right\|_{\Ls{p}{\left[-\frac12,\frac12\right]}}  \right).
  \end{equation*}
  }
  \end{theoA}

  As a corollary of Theorem~\ref{cor1R} and Theorem~\ref{TL} we get
\begin{lemma}\label{le1}
  Suppose that $1 \le p \le q \le\infty,$ $(p,q)\neq (\infty,\infty),$
  $a, b >-\frac1q,$ $c,d>-\frac1p,$
$r$ is a nonnegative integer$,$ $k$ is a positive integer$,$ and
  $$
     h=k-\left(\frac1p-\frac1q\right).
  $$
      Then$,$ there exists a~constant $C=C(p,q,a,b,c,d,r,k)$
 such that for any function $f$
with  $f^{(r+k-1)}$ absolutely continuous on $(-1,1)$ and  $f^{(r+k)} \w{a+h}{b+h}\in L_p$
we have
   \begin{equation}\label{MH1}
     \left\| f^{(r)} \w{a}{b}  \right\|_{q}\le
     C\left(\left\|f \w{c}{d}  \right \|_{p}+
     \left\|f^{(r+k)} \w{a+h}{b+h}  \right\|_{p}\right).
  \end{equation}
 Inequality~\eqref{MH1} is sharp in the following sense.
 If $a-c<r+\left(\frac1p-\frac1q\right) ,$ then for any $\varepsilon>0$  there exists  $\{f_n\}\subset C^{k+r}[-1,1]$
 such that
 \begin{equation}
 \label{f7}
  \left\|  f_n^{(r)} \w{a}{b} \right\|_{q}\cdot
 \left( \left\|   f_n \w{c}{d} \right\|_{1}+
 \left\| f_n^{(r+k)} \w{a+h+\varepsilon}{b+h} \right\|_{p}\right)^{-1}\to\infty\quad\mbox{as}\quad n\to \infty.
  \end{equation}
  The analogous statement also holds with respect to  the parameter $b$.
  \end{lemma}
\begin{proof}[Proof of Lemma~\ref{le1}]
It is enough to verify inequality~\eqref{MH1} for $k=1$. The proof in the general case is by induction on $k$.
    Note that $f^{(r)}$ is continuous on $\left[-\frac12,\frac12\right]$ by our assumption. We take
   $\overline{x}\in \left[-\frac12,\frac12\right]$ such that
  $$
  \left|f^{(r)}(\overline{x})\right|= \min\Big\{\left|f^{(r)}(x)\right|: x\in \left[-\tfrac12,\tfrac12\right]\Big\}.
  $$
   Let   $g(x)=f^{(r)}(x)-f^{(r)}(\overline{x})$, then
   \begin{align*}
   \left\| f^{(r)}\w{a}{b}  \right\|_{q}&
   \le \left\| g\w{a}{b}  \right\|_{q}+\left|f^{(r)}(\overline{x})\right|  \left\| \w{a}{b}\right\|_{q}
   \\
    &\le \left\| g\w{a}{b}  \right\|_{\Ls{q}{[-1,\overline{x}]}}
    +\left\|g \w{a}{b} \right\|_{\Ls{q}{[\overline{x},1]}}+\left|f^{(r)}(\overline{x})\right|  \left\| \w{a}{b}\right\|_{\Ls{q}{[-1,1]}}.
    \end{align*}
    To estimate the first term, we  apply Theorem~\ref{cor1R}
     (for the interval $[-1,\overline{x}]$ instead of $[0,\overline{x}]$) with ${h=1-\left(\frac1p-\frac1q\right)}$:
    \begin{align*}
    \left\|  g\w{a}{b} \right\|_{\Ls{q}{[-1,\overline{x}]}}&
    \le 2^{|a|}\left\| g(x)(1+x)^{b} \right\|_{\Ls{q}{[-1,\overline{x}]}} \le
     2^{|a|} C\left\| g'(x)(1+x)^{b+h} \right\|_{\Ls{q}{[-1,\overline{x}]}}
    \\
    &\le 2^{|a|+|a+h|} C
    \left\|g'(x) (1-x)^{a+h} (1+x)^{b+h} \right\|_{\Ls{p}{[-1,\overline{x}]}}
    \\
    &\le 2^{|a|+|a+h|} C \left\|g' \w{a+h}{b+h} \right\|_{\Ls{p}{[-1,1]}}
    = 2^{|a|+|a+h|} C \left\|
    f^{(r+1)}
    \w{a+h}{b+h} \right\|_{\Ls{p}{[-1,1]}}.
    \end{align*}
 A similar estimate holds for $\left\| g\w{a}{b}\right\|_{\Ls{q}{[\overline{x},1]}}$ as well.

    To estimate $\left|f^{(r)}(\overline{x})\right|$, we apply Theorem~\ref{TL}:
    \begin{align*}
    \left|f^{(r)}(\overline{x})\right|&\le \left\|f^{(r)}\right\|_{\Ls{1}{ \left[-\tfrac12,\tfrac12\right]}} \le
    C\left(
           \left\|f\right\|_{\Ls{1}{ \left[-\tfrac12,\tfrac12\right]}}+\left\|f^{(r+1)}\right\|_{\Ls{1}{ \left[-\tfrac12,\tfrac12\right]}}
       \right)
      \\
    &\le  2^{|c|+|d|+|a+h|+|b+h|}  C
    \left(\left\|f \w{c}{d}\right\|_{\Ls{p}{[-1,1]}}+
    \left\|f^{(r+1)} \w{a+h}{b+h}\right\|_{\Ls{p}{[-1,1]}}
    \right),
    \end{align*}
    where $C$ depends only on $r+1$.
Thus, \eqref{MH1} follows.

Let us now show~\eqref{f7}. Since for any $0\le \varepsilon_1\le \varepsilon_2$ the estimate
  $$  \w{a+h+\varepsilon_2}{b+h}(x)\le 2^{\varepsilon_2-\varepsilon_1}
   \w{a+h+\varepsilon_1}{b+h}(x), \qquad x\in[-1,1],
  $$ holds,
 we can assume
  \begin{equation}\label{varepsc}
0<\varepsilon \le c-a+r+1/p-1/q.
 \end{equation}
 For  $m>r+k$, consider the sequence of functions
$$
 f_n(x)=\big((x+1/n-1)_+\big)^{m}, \quad  x\in\left[-1,1\right],
 \qquad
 y_+=\max\{y,0\}.
 $$
 It is easy to verify that
  if $\mu\ge 0$ and  $\nu > -1/q$, then
 $$
 \left\| \big((1/n-1+x)_+\big)^\mu (1-x)^\nu \right\|_{q}
 \asymp \frac{1}{n^{\mu+\nu+1/q}}\quad \mbox{ as }\quad  n\to \infty.
 $$
 Here  $A_n\asymp B_n$ as $n\to \infty$ means that
 $B_n/C\le A_n\le C B_n$ for some positive constant $C$ and all $n$.
 Using this, we get
 \begin{align*}
 &\left\| f_n \w{c}{d} \right\|_{p} \asymp \frac{1}{n^{m+c+1/p}},
 \qquad
  \left\| f_n^{(r)}  \w{a}{b} \right\|_{q} \asymp
    \frac{1}{n^{m-r+a+1/q}},
\\
 &\left\| f_n^{(r+k)} \w{a+h+\varepsilon}{b+h}\right\|_{p}\asymp
  \frac{1}{n^{m-r-k+a+h+\varepsilon+1/p}}=\frac{1}{n^{m-r+a+\varepsilon+1/q}}.
 \end{align*}
 Under  assumption~\eqref{varepsc} we have
  $$
  \left\| f_n  \w{c}{d} \right\|_{p} +\left\| f_n^{(r+k)} \w{a+h+\varepsilon}{b+h} \right\|_{p} \asymp
  \frac{1}{n^{m-r+a+\varepsilon+1/q} }, $$
 and therefore,
$$
 \frac{ \left\| f_n^{(r)} \w{a}{b} \right\|_{q}}
       {\left\| f_n  \w{c}{d} \right\|_{p} +\left\| f_n^{(r+k)} \w{a+h+\varepsilon}{b+h} \right\|_{p}}
  \asymp  n^\varepsilon  \quad \mbox{ as } \quad n\to \infty,
 $$
concluding the proof.
   \end{proof}


   \subsection{Hardy--Littlewood type inequalities}
To prove Hardy--Littlewood type inequalities for the fractional integral $\Oi^{(\alpha,\beta)}_{\sigma}$, we will use  the Muckenhoupt transplantation theorem \cite[Collorary 17.11]{Mu-1987}, which is  written in our notation as follows.
   \begin{theoA}\label{ThMu}{\it
   If $1<\overline{p}\le\overline{ q}<\infty,$ \
   ${\overline{\alpha},\,\overline{\beta},\, \overline{\gamma},\, \overline{\delta}>-1},$ \
   $\ds {\overline{a},\,\overline{b},\,}\overline{c},\,\overline{d} >-1,$ \
   $$
      s=\frac1{\overline{p}}-\frac 1{\overline{q}},
  $$
   $$
      \frac{\overline{a}}{\overline{q}}=\frac{\overline{c}}{\overline{p}}+\frac{\overline{\alpha}-\overline{\gamma}}{2}+
     \frac12\left(\frac1{\overline{p}}-\frac1{\overline{q}}\right),  \quad
      \frac{\overline{b}}{\overline{q}}=\frac{\overline{d}}{\overline{p}}+\frac{\overline{\beta}-\overline{\delta}}{2}+
       \frac12\left(\frac1{\overline{p}}-\frac1{\overline{q}}\right),
   $$
   the quantities
   $\overline{ A}=(\overline{c}+1)/\overline{p}-\overline{\gamma}$   and
   $\overline{ B}=(\overline{d}+1)/\overline{p}-\overline{\delta} $
  are not positive integers, $M=\max\{0,[\overline{A}]\}$, $N=\max\{0,[\overline{B}]\}$,
  $f \in \LL{\overline{p}}{\overline{c},\overline{d}}$,
   $$
   {\widehat{f}}^{(\overline{\gamma},\overline{\delta})}_k=0, \quad 0\le k\le M+N-1,
  $$
  $h$ is an integer$,$ $\cm{k}$ has the form
  \begin{equation*}
    \cm{k}=\sum_{j=0}^{J-1}c_j(k+1)^{-s-j}+O\left((k+1)^{-s-J}\right)
 \end{equation*}
  with $J\ge \overline{\alpha}+\overline{\beta}+\overline{\gamma}+\overline{\delta}+6+2M+2N$
  and $0\le \rho<1$, then
  \begin{equation*}
  T_\rho f(x)=\sum_{k=0}^\infty \rho^k \cm{k}
  \widehat{f}_k^{(\overline{\gamma},\overline{\delta})}
   \psi^{(\overline{\alpha},\overline{\beta})}_{k+h}(x)
  \end{equation*}
  converges for every $x \in(-1,1),$
   \begin{equation*}
   \left\|T_\rho f \right\|_{\Lw{\overline{q}}{\overline{a},\overline{b}}} \le C
   \left\|f \right\|_{\Lw{\overline{p}}{\overline{c},\overline{d}}} ,
   \end{equation*}
   where
   $C$ is independent of $\rho$ and $f$. Moreover$,$
  there is a function $Tf$ in $\LL{\overline{q}}{\overline{a},\overline{b}}$ such that $T_\rho f$ converges to $Tf$
  in $\LL{\overline{q}}{\overline{a},\overline{b}}$ as $\rho\to 1-$.
  If it is also assumed that
  {$ \overline{a}+1<(\overline{\alpha}+1)\overline{q}$}
  and
  {$\overline{b}+1<(\overline{\beta}+1)\overline{q},$}
  then
  \begin{equation*}
  \widehat{Tf}^{(\overline{\alpha},\overline{\beta})}_k
  =\begin{cases}
  0, & 0\le k \le h-1 \\
  \cm{k-h}
  \widehat{f}^{(\overline{\gamma},\overline{\delta})}_{k-h} , & \max(0,h)\le k.
  \end{cases}
 \end{equation*}
  }
  \end{theoA}

The next Hardy--Littlewood inequality
is a simple corollary of Theorem~\ref{ThMu}.
 {
 \begin{corollary}\label{r2}
  Let   $1 < p < q <\infty,$  $-1/2\ge a \ge b > -1,$ $ \alpha \ge \beta>-1,$
  {$(a+1)<(\alpha+1)p,$ $(b+1)<(\beta+1)p,$} and
   $$
    \sigma \ge \frac{1}{p}-\frac1q.
   $$
Let also {$f\in \LL{p}{a,b}.$}
  Then  there exists $C$  independent of $f$ such that
  \begin{equation}\label{hard-li}
  \left\|\Oi^{(\alpha,\beta)}_{\sigma}f\right\|_{\Lw{q}{a,b}}\le C
  \big\| f     \vphantom{I^{(\alpha,\beta)}_{\sigma}}    \big\|_{\Lw{p}{a,b}}.
  \end{equation}
 \end{corollary}
}

In the special case $(\alpha,\beta)=(a,b)$,
 the Hardy--Littlewood inequality (\ref{hard-li})
 was studied
 by Askey and Wainger \cite[Sec. J]{As-Wa-1969} (see also \cite{ask}) and later by
  Bavinck and Trebels  \cite[Theorem 5.4]{Ba-1972}, \cite[Theorems 1 and 1']{Ba-Tr-1978}.

  \begin{theoA}[\textrm{\cite{As-Wa-1969, Ba-Tr-1978}}] \label{ThBa}
  {\it Let $1<p<q<\infty,$  $a\ge b >-1,$ $a+b\ge-1,$
  and
 \begin{equation*}\label{condsigma}
  \sigma \ge (2a+2)\left(\frac{1}{p}-\frac{1}{q}\right) .
 \end{equation*}
 If $f\in \LL{p}{a,b},$ then $\Oi_\sigma^{(a, b)} f\in \LL{q}{a, b}$ and
  $$
  \left\| \Oi_\sigma^{(a, b)} f \right\|_{\Lw{q}{a, b}}\le C(p,q,a, b)
  \big\|f\big\|
  _{\Lw{p}{a, b}}.
  $$}
 \end{theoA}

 For $(\alpha,\beta)\neq (a,b)$ we have the following result.

\begin{theorem}\label{thHI}
  Let   $1 < p < q <\infty,$  $a \ge b >-1,$ $a\ge-1/2,$  $ \alpha \ge \beta>-1,$
 \begin{equation}\label{thHI_02}
    p(\alpha-\beta)\le 2(a-b)\le q(\alpha-\beta),
 \end{equation}
  the quantities $ A=(a+1)/p-\alpha$ and
$B=(b+1)/p-\beta$ be not positive integers$,$
and either $\alpha=a,$ or $\alpha > a$ and $q > 2,$
or $\alpha < a$ and $p < 2.$
Let     \begin{equation}\label{thHI_03}
  \sigma\ge (2a+2)\left(\frac1p-\frac 1q\right),
  \end{equation}
 $f\in \LL{p}{a,b}\cap  \LL{1}{\alpha,\beta}$ and
  \begin{equation}\label{thHI_04}
     {\widehat{f}}^{(\alpha,\beta)}_k=0, \quad 0\le k\le \max\left\{0,[A]\right\}+\max\left\{0,[B]\right\}-1.
  \end{equation}
  Then  there exists $C$  independent of $f$ such that
  \begin{equation}\label{thHI_05}
  \left\|\Oi^{(\alpha,\beta)}_{\sigma}f\right\|_{\Lw{q}{a,b}}\le C
  \big\| f     \vphantom{I^{(\alpha,\beta)}_{\sigma}}    \big\|_{\Lw{p}{a,b}}.
  \end{equation}
  \end{theorem}

  \begin{proof}
    It is sufficient to prove this theorem for polynomials.
    Indeed, suppose that \eqref{thHI_05} holds for polynomials.
    Consider a sequence of polynomials $\{Q_m\}$ convergent to $f$ in $\LL{p}{a,b}$ and
    $\LL{1}{\alpha,\beta}$.
    Then  $\{\Oi_\sigma^{(\alpha,\beta)}Q_m\}$ is a Cauchy sequence in $\LL{q}{a,b}$ and it converges
    to some function $g$  in $\LL{q}{a,b}$.  Without loss of generality we can assume that
     $\{\Oi_\sigma^{(\alpha,\beta)}Q_m\}$  converges  to $g$ a.e. on $[-1,1]$.
     Since the operator $\Oi_\sigma^{(\alpha,\beta)}$ is continuous in $\LL{1}{\alpha,\beta}$, the sequence
    $\{\Oi_\sigma^{(\alpha,\beta)}Q_m\}$ converges   to $\Oi_\sigma^{(\alpha,\beta)}f$  in
    $\LL{1}{\alpha,\beta}$.
    There is a subsequence $\{\Oi_\sigma^{(\alpha,\beta)}Q_{m_j}\}$ convergent to $\Oi_\sigma^{(\alpha,\beta)}f$ a.e. on $[-1,1]$.
 Therefore, $g=\Oi_\sigma^{(\alpha,\beta)}f$.

   Let $f$ be a polynomial, i.e.,
   $$
   f =\sum_{k=0}^\infty c_k \psi^{(\alpha,\beta)}_k,
   $$
   where $c_k={\widehat{f}}^{(\alpha,\beta)}_k$ and $c_k=0$ for $k> \deg (f)$.  

  \underline{Case 1.} Consider  $\alpha \ge a,$ $q \ge 2.$
   More precisely, under assumption of the theorem, the following relations are possible:
    $\alpha > a $ and $q > 2$ or  $ \alpha=a$ and $q\ge 2$.

   Now, we define $\alpha_1$ and  $p_1$.
   If  $\alpha>a,$ then we set
  \begin{align*}
    &\alpha_1=\frac{q\alpha-2a}{q-2},
    \\
    &\frac{\alpha_1}{p_1}=\frac{a}{p}+\frac{\alpha_1-\alpha}{2}+\frac{1}{2}\left(\frac1{p}-\frac1{p_1}\right).
     \end{align*}
In this case, we have
  $$    \frac{2\alpha_1+1}{p_1}=\frac{2a+1}{p}+\frac{2(\alpha-a)}{q-2}
    $$
and
\begin{equation}\label{thHI_p}
 (2\alpha_1+2)\left(\frac{1}{p_1}-\frac1q \right)+\frac1p-\frac{1}{p_1}=
(2a+2)\left(\frac1p-\frac1q \right).
\end{equation}
Notice that condition $\alpha>a$ implies that
   $\alpha_1 > \max\{ a,\alpha, 0\}$ and $p<p_1<q$.
   
   If  $\alpha=a,$ then we set $\alpha_1=\alpha,$ $p_1=p.$

 \vskip 0.3cm

      We divide the rest of the proof in Case 1 into three steps.
       \vskip 0.2cm

\noindent  \underline{Step 1.1}.
   We apply Theorem~\ref{ThMu}
   with $(\overline{q},\overline{p})=(p_1,p)$,  $(\overline{\alpha},\overline{\beta})=(\alpha_1,\alpha_1)$,
   $(\overline{\gamma},\overline{\delta})=(\alpha,\beta)$,
   $(\overline{c},\overline{d})=(a,b)$, $h=0$,
    $s=\sigma_1=\frac{1}{p}-\frac1{p_1}$,
    and
   $$
    \cm{k}=\left( \lambda_k^{(\alpha_1,\alpha_1)}\right)^{-\sigma_1}.
   $$
   Then we have
   $ \overline{a}=\alpha_1,$
   \begin{equation}\label{thHI_1b}
   \frac{\overline{b}}{p_1}=\frac{b}{p}+\frac{\alpha_1-\beta}{2}+\frac12\left(\frac1p-\frac1{p_1}\right)=
   \frac{\alpha_1}{p_1}-\frac{2(a-b)-p(\alpha-\beta)}{2p},
   \end{equation}
   $$
   A=\frac{a+1}{p}-\alpha , \quad  B=\frac{b+1}{p}-\beta.
   $$
    Therefore, under condition~\eqref{thHI_04}  for any $\rho\in(0,1)$, we
    obtain the inequality
   \begin{equation}\label{th2_011}
    \left\|c_0 +\sum_{k=1}^\infty \rho^k\left(\lambda_k^{(\alpha_1,\alpha_1)}\right)^{-\sigma_1}
   c_k \psi^{(\alpha_1,\alpha_1)}_k\right\|_{\Lw{p_1}{\alpha_1,\overline{b}}} \le C \|f\|_{\Lw{p}{a,b}},
   \end{equation}
   where $C$ is independent of $f$ and $\rho$.
    Since $f$ is a polynomial, the sum is finite, and  we can rewrite~\eqref{th2_011} as
   \begin{equation*}
    \left\|c_0+\sum_{k=1}^\infty \left(\lambda_k^{(\alpha_1,\alpha_1)}\right)^{-\sigma_1}
   c_k \psi^{(\alpha_1,\alpha_1)}_k\right\|_{\Lw{p_1}{\alpha_1,\overline{b}}} \le C \|f\|_{\Lw{p}{a,b}}.
   \end{equation*}
   Relations  \eqref{thHI_02} and \eqref{thHI_1b} show that $\alpha_1\ge \overline{b}$, and hence,
   \begin{equation}\label{th2_05}
   \left\|c_0+\sum_{k=1}^\infty \left(\lambda_k^{(\alpha_1,\alpha_1)}\right)^{-\sigma_1}
   c_k \psi^{(\alpha_1,\alpha_1)}_k\right\|_{\Lw{p_1}{\alpha_1,\alpha_1}} \le C \|f\|_{\Lw{p}{a,b}}.
   \end{equation}
\vskip 0.2cm

\noindent
  \underline{Step 1.2}. In view of~\eqref{thHI_03} and \eqref{thHI_p}, we have
     $$
    \sigma-\sigma_1\ge (2\alpha_1+2)\left(\frac{1}{p_1}-\frac{1}{q} \right),
    $$ we can apply  Theorem~\ref{ThBa}
    for the pair of spaces $\LL{q}{\alpha_1,\alpha_1}$
   and
   $\LL{p_1}{\alpha_1,\alpha_1}$
 to get
    \begin{equation}\label{th2_06}
     \left\|c_0+\sum_{k=1}^\infty \left(\lambda_k^{(\alpha_1,\alpha_1)}\right)^{-\sigma}
   c_k \psi^{(\alpha_1,\alpha_1)}_k   \right\|_{\Lw{q}{\alpha_1,\alpha_1}}\!\!\!
    \le C  \left\|c_0+\sum_{k=1}^\infty \left(\lambda_k^{(\alpha_1,\alpha_1)}\right)^{-\sigma_1}
   c_k \psi^{(\alpha_1,\alpha_1)}_k\right\|_{\Lw{p_1}{\alpha_1,\alpha_1}}.
   \end{equation}
   \vskip 0.2cm

\noindent
  \underline{Step 1.3}.
   We use Theorem~\ref{ThMu} once again
   with $(\overline{q},\overline{p})=(q,q)$,  $(\overline{\alpha},\overline{\beta})
   =(\alpha,\beta)$,
   $(\overline{\gamma},\overline{\delta})=(\alpha_1,\alpha_1)$,
   $(\overline{c},\overline{d})=(\alpha_1,\alpha_1)$,
   and
   $$
    \cm{k}=\left(\lambda_k^{(\alpha,\beta)}/ \lambda_k^{(\alpha_1,\alpha_1)}\right)^{-\sigma}.
   $$
   Then $s=0$, $\overline{a}=a,$
   \begin{equation}\label{thHI_3b-}
   \frac{\overline{b}}{q}=\frac{\alpha_1}{q}+\frac{\beta-\alpha_1}{2}=
   \frac{b}{q}-\frac{q(\alpha-\beta)-2(a-b)}{2q},
   \end{equation}
   and
   $$
   A=B=\frac{\alpha_1+1}{q}-\alpha_1=
\alpha_1\left( \frac1q-1\right)+\frac1q \le
  -\frac12\left( \frac1q-1\right)+\frac1q < 1, \quad [A]=[B]=0.
   $$
    We have
    \begin{equation*}
     \left\|  c_0+ \sum_{k=1}^\infty\left(\lambda_k^{(\alpha,\beta)}\right)^{-\sigma}
               c_k \psi^{(\alpha,\beta)}_k
        \right\|_{\Lw{q}{a,\overline{b}}}\le
   C\left\|c_0+\sum_{k=1}^\infty \left(\lambda_k^{(\alpha_1,\alpha_1)}\right)^{-\sigma}
   c_k \psi^{(\alpha_1,\alpha_1)}_k \right\|_{\Lw{q}{\alpha_1,\alpha_1}}.
   \end{equation*}
    Relations  \eqref{thHI_02} and \eqref{thHI_3b-} show  that $\overline{b}\le b,$ and hence,
        \begin{equation}\label{th2_07}
       \left\|  c_0+ \sum_{k=1}^\infty\left(\lambda_k^{(\alpha,\beta)}\right)^{-\sigma}
               c_k \psi^{(\alpha,\beta)}_k
        \right\|_{\Lw{q}{a,b}}\le 2^{b-\overline{b}}
     \left\|  c_0+ \sum_{k=1}^\infty\left(\lambda_k^{(\alpha,\beta)}\right)^{-\sigma}
               c_k \psi^{(\alpha,\beta)}_k
        \right\|_{\Lw{q}{a,\overline{b}}}.
   \end{equation}
 Finally, combining \eqref{th2_05}, \eqref{th2_06}, and
\eqref{th2_07},  we obtain inequality \eqref{thHI_05}.

\bigskip

\underline{Case 2.} Consider  $\alpha \le a,$ $p \le 2.$
   More precisely, under assumption of the theorem, the following relations are possible:
    $\alpha < a $ and $p< 2$ or  $ \alpha=a$ and $p\le 2$.

   Now, we define $\alpha_1$ and  $q_1$.
   If  $\alpha <a,$ then we set
  \begin{align*}
    &\alpha_1=\frac{2a-p\alpha}{2-p},
    \\
    &\frac{a}{q}=\frac{\alpha_1}{q_1}+\frac{\alpha-\alpha_1}{2}+
    \frac{1}{2}\left(\frac1{q_1}-\frac1{q}\right).
     \end{align*}
In this case, we have
   $$    \frac{2\alpha_1+1}{q_1}=\frac{2a+1}{q}+\frac{2(a-\alpha)}{2-p}
    $$
and
\begin{equation}\label{thHI_2p}
 (2\alpha_1+2)\left(\frac1p-\frac1{q_1} \right)+\frac1{q_1}-\frac{1}{q}=
(2a+2)\left(\frac1p-\frac1q \right).
\end{equation}
Notice that condition $\alpha <a$ implies that
   $\alpha_1 > \max\{ a,\alpha, 0\}$ and $p<q_1<q$.

   If  $\alpha=a,$ then we set $\alpha_1=\alpha,$ $q_1=q.$

 \vskip 0.3cm

We can argue similarly to the proof in Case 1 dividing the rest of the proof  into three steps.
       \vskip 0.2cm

\noindent  \underline{Step 2.1}.
   We are going to use Theorem~\ref{ThMu}
   with $(\overline{q},\overline{p})=(p,p)$,
  $(\overline{\alpha},\overline{\beta})=(\alpha_1,\alpha_1)$,
   $(\overline{\gamma},\overline{\delta})=(\alpha,\beta)$,
   $(\overline{c},\overline{d})=(a,b)$, $h=0$,
    $s=0$,
    and
   $
    \cm{k}=1.
   $
Then 
   $ \overline{a}=\alpha_1,$
   \begin{equation}\label{thHI_21b}
   \frac{\overline{b}}{p}=\frac{b}{p}+\frac{\alpha_1-\beta}{2}=
   \frac{\alpha_1}{p}-\frac{2(a-b)-p(\alpha-\beta)}{2p},
   \end{equation}
   $$
   A=\frac{a+1}{p}-\alpha , \quad  B=\frac{b+1}{p}-\beta.
   $$
    Therefore, under condition~\eqref{thHI_04}  for any $\rho\in(0,1)$, we
    obtain the inequality
   \begin{equation}\label{thHI_2_1}
    \left\|c_0 +\sum_{k=1}^\infty \rho^k   c_k \psi^{(\alpha_1,\alpha_1)}_k\right\|_{\Lw{p}{\alpha_1,\overline{b}}} \le C \|f\|_{\Lw{p}{a,b}},
   \end{equation}
   where $C$ does not depend on $f$ and $\rho$.
    Since $f$ is a polynomial, the sum is finite.
    Taking into account  \eqref{thHI_02} and \eqref{thHI_21b},
    we conclude that $\alpha_1\ge \overline{b}$, and hence,
      and  we can rewrite~\eqref{thHI_2_1} as
   \begin{equation}\label{thHI_2_2}
   \left\|c_0+\sum_{k=1}^\infty
   c_k \psi^{(\alpha_1,\alpha_1)}_k\right\|_{\Lw{p}{\alpha_1,\alpha_1}} \le C \|f\|_{\Lw{p}{a,b}}.
   \end{equation}
\vskip 0.2cm

\noindent
  \underline{Step 2.2}.
 Set $\sigma_1=\sigma-\left(\frac1{q_1}-\frac1{q} \right)$.
In view of~\eqref{thHI_03} and \eqref{thHI_2p}, we have
     $$
    \sigma_1\ge (2\alpha_1+1)\left(\frac{1}{p}-\frac{1}{q_1}\right).
    $$
 We  can apply  Theorem~\ref{ThBa}
    for the pair of spaces $\LL{q_1}{\alpha_1,\alpha_1}$
   and
   $\LL{p}{\alpha_1,\alpha_1}$
 to get
    \begin{equation}\label{thHI_2_3}
     \left\|c_0+\sum_{k=1}^\infty \left(\lambda_k^{(\alpha_1,\alpha_1)}\right)^{-\sigma_1}
   c_k \psi^{(\alpha_1,\alpha_1)}_k   \right\|_{\Lw{q_1}{\alpha_1,\alpha_1}}\!\!\!
    \le C  \left\|c_0+\sum_{k=1}^\infty  c_k \psi^{(\alpha_1,\alpha_1)}_k\right\|_{\Lw{p}{\alpha_1,\alpha_1}}.
   \end{equation}
   \vskip 0.2cm

\noindent
  \underline{Step 2.3}.
  We use Theorem~\ref{ThMu} once again
   with $(\overline{q},\overline{p})=(q,q_1)$,  $(\overline{\alpha},\overline{\beta})
   =(\alpha,\beta)$,
   $(\overline{\gamma},\overline{\delta})=(\alpha_1,\alpha_1)$,
   $(\overline{c},\overline{d})=(\alpha_1,\alpha_1)$,
   and
   $$
    \cm{k}=\left(\lambda_k^{(\alpha,\beta)}\right)^{-(\sigma-\sigma_1)}
\left( \lambda_k^{(\alpha_1,\alpha_1)}/\lambda_k^{(\alpha,\beta)}\right)^{\sigma_1}.
   $$
 Hence, $s=\sigma-\sigma_1=\frac1{q_1}-\frac1q$, $\overline{a}=a,$
   \begin{equation}\label{thHI_3b}
   \frac{\overline{b}}{q}=\frac{\alpha_1}{q_1}+\frac{\beta-\alpha_1}{2}+\frac{1}{2}\left(\frac{1}{q_1}-\frac{1}{q}\right)=
   \frac{b}{q}-\frac{q(\alpha-\beta)-2(a-b)}{2q},
   \end{equation}
   and
   $$
   A=B=\frac{\alpha_1+1}{q_1}-\alpha_1=
\alpha_1\left( \frac1{q_1}-1\right)+\frac1{q_1} \le
  -\frac12\left( \frac1{q_1}-1\right)+\frac1{q_1} < 1, \quad [A]=[B]=0.
   $$
    We have
    \begin{equation*}
     \left\|  c_0+ \sum_{k=1}^\infty\left(\lambda_k^{(\alpha,\beta)}\right)^{-\sigma}
               c_k \psi^{(\alpha,\beta)}_k
        \right\|_{\Lw{q}{a,\overline{b}}}\le
   C\left\|c_0+\sum_{k=1}^\infty \left(\lambda_k^{(\alpha_1,\alpha_1)}\right)^{-\sigma_1}
   c_k \psi^{(\alpha_1,\alpha_1)}_k \right\|_{\Lw{q_1}{\alpha_1,\alpha_1}}.
   \end{equation*}
    Taking into account  \eqref{thHI_02} and \eqref{thHI_3b},
    we see that $\overline{b}\le b,$ and hence,
        \begin{equation}\label{thHI_2_4}
       \left\|  c_0+ \sum_{k=1}^\infty\left(\lambda_k^{(\alpha,\beta)}\right)^{-\sigma}
               c_k \psi^{(\alpha,\beta)}_k
        \right\|_{\Lw{q}{a,b}}\le 2^{b-\overline{b}}
     \left\|  c_0+ \sum_{k=1}^\infty\left(\lambda_k^{(\alpha,\beta)}\right)^{-\sigma}
               c_k \psi^{(\alpha,\beta)}_k
        \right\|_{\Lw{q}{a,\overline{b}}}.
   \end{equation}
 Finally, combining \eqref{thHI_2_2}, \eqref{thHI_2_3}, and
\eqref{thHI_2_4},  we obtain inequality \eqref{thHI_05}.

\end{proof}


\medskip
 \section{Ulyanov-type inequalities for $K$-functionals}

 Definitions and facts, given in this section and in the next one, are based on the  books \cite{book, Ma-Mi}; see also
 \cite{Da-Di-2005, Di-1998} and the recent survey   \cite{Di-2007}.

 In this section, we assume that $1 \le p \le \infty$, $a,b >-1$, $\alpha,\beta > -1$ and
 \begin{equation}\label{abab}
 \frac{a+1}{p} - \alpha <1 , \quad  \frac{b+1}{p} - \beta <1.
 \end{equation}
 Then,
 since  $\LL{p}{a,b} \subset \LL{1}{\alpha,\beta}$,
 the Fourier--Jacobi expansion \eqref{falphabeta1} is well-defined for any
 ${f\in \LL{p}{a,b}}$.

  Denote by $\PP_{n}$  the set of all algebraic polynomials of
 degree at most $n$, $\PP=\cup_{n\ge 0}\PP_n$.
  Let
     $P_{n,f}=P_n(f)_{\Lw{p}{a,b}},$ $P_{n,f}\in \PP_n$, be a near best polynomial approximant of a function  $f\in \LL{p}{a,b}$, that is,
 \begin{equation}\label{Pnf}
   \|f-P_{n,f}\|_{\Lw{p}{a,b}}\le C E_n(f)_{\Lw{p}{a,b}}, \quad   E_n(f)_{\Lw{p}{a,b}}=\inf\left\{\|f-P\|_{\Lw{p}{a,b}}:\  P\in \PP_n\right\}.
  \end{equation}

The $K$-functional corresponding to the differential operator~$\Opp^{(\alpha,\beta)} $ and a real positive number $r$
 is defined by
  \begin{align}\label{K-func1}
   &K^r(f, \Opp^{(\alpha,\beta)}_r,t)_{\Lw{p}{a,b}}=
  \inf \Big\{\|f-g\|_{\Lw{p}{a,b}}+t^{r}\|\Opp^{(\alpha,\beta)}_{r} g\|_{\Lw{p}{a,b}}: \  g\in W^{r, (\alpha,\beta)}_{p, (a,b)}  \Big\}
  \end{align}
(see  \cite[(1.9)]{Di-1998}),
 where
 $W^{r, (\alpha,\beta)}_{p, (a,b)}
 =\Big\{g: g,\,\Opp^{(\alpha,\beta)}_{r} g\in \LL{p}{a,b}\Big\}.$
 The following  realization result holds:
  \begin{equation}\label{K-func2}
   K^r\left(f,\Opp^{(\alpha,\beta)}_r , 1/n\right)_{\Lw{p}{a,b}}
  \asymp
 \|f-P_{n,f}\|_{\Lw{p}{a,b}}+
   n^{-r}\|\Opp^{(\alpha,\beta)}_r  P_{n,f}\|_{\Lw{p}{a,b}},
  \ 1<p<\infty.
  \end{equation}
 It is a corollary of Theorem 6.2 in \cite{Di-1998}.
    To apply this theorem, we have to show that the Ces\`{a}ro  operator $C^\ell_n$ given by
    $$
     C^\ell_n(f)=
    \sum_{k=0}^n
    \left(1-\tfrac{k}{n+1}\right)\left(1-\tfrac{k}{n+2}\right)\cdots \left(1-\tfrac{k}{n+\ell}\right) \widehat{f}_k \psi^{(\alpha,\beta)}_k
   $$
   is bounded in  $\LL{p}{a,b}$	for some $\ell$. This fact is mentioned in  \cite[Sec. 3]{Da-Di-2005}.
   Moreover, from \cite[Theorem~1.10, p.~4]{Mu-1987} (see also \cite[Theorem M]{Da-Di-2005})
   it easily follows that
    the operator $C^\ell_n$ is bounded in $\LL{p}{a,b}$ for any
   \begin{align*}
    \ell >\max &\left\{  \left|\tfrac{2(a+1)}{p}-\alpha-1 \right|,
                             \left|\tfrac{2(b+1)}{p}-\beta-1 \right|, \right.\\
                             &\left.\left|\tfrac{2(a+1)}{p}-\alpha-\tfrac12-\tfrac1p \right|,
                             \left|\tfrac{2(b+1)}{p}-\beta-\tfrac12-\tfrac1p \right|,
                             \left|\tfrac{2}{p}(a-b)-(\alpha-\beta) \right|
                   \right\}.
 \end{align*}
 Note that one can equivalently consider the boundedness of the Riesz means, see \cite[Theorem~3.19]{lecture}.

Now we formulate and prove the main result -- Ulyanov type inequality for $K$-functionals with Jacobi weights.
 Theorem~\ref{thU1} contains Theorem~\ref{thU1short}, stated in Introduction, as a particular case.

 \begin{theorem} \label{thU1}
{\it
  Let   $1< p<q < \infty$ and $r>0$. Suppose that $\alpha,\beta >-1,$ $a \ge b>-1$, $a\ge-1/2,$
 inequalities  \eqref{abab} hold$,$ and either $(\alpha,\beta)=(a,b),$ or
 $$
    p(\alpha-\beta)\le 2(a-b)\le q(\alpha-\beta),
$$
and  $\alpha=a,$ or $\alpha > a,$ $q > 2,$
or $\alpha < a,$ $p < 2.$

  Suppose also that
 $$\sigma=(2a+2)\left(\frac1p-\frac1q\right).$$
If~$f\in \LL{p}{a,b}$ and
 $$
 \int_0^1 \left(u^{-\sigma} K^{r+\sigma}(f, \Opp^{(\alpha,\beta)}_{r+\sigma},u)_{\Lw{p}{a,b}} \right)^{q} \frac{du}{u}<\infty,
 $$
 then $f\in \LL{q}{a,b}$
and
  \begin{equation}\label{thU1_01}
  K^r(f, \Opp^{(\alpha,\beta)}_r,t)_{\Lw{q}{a,b}}\le C
     \left(
    \int_0^t \left(u^{-\sigma} K^{r+\sigma}(f, \Opp^{(\alpha,\beta)}_{r+\sigma},u)_{\Lw{p}{a,b}} \right)^{q}
    \frac{du}{u}
    \right)^{1/q}.
   \end{equation}
  }
\end{theorem}

 Theorem~\ref{thU1} extends the  results of~\cite[Theorem 11.2]{Di-Ti-2005} and
\cite[Section 3.3.1]{trebels-jat}
 in two  directions.
 First, our estimate involves the  $K$-functional of order $r+\sigma$, i.e., we get the sharp estimate. Second, we consider
 the case when
 $(\alpha,\beta)\neq (a,b)$.
We also remark that the sharp Ulyanov inequality for functions on $\mathbb{S}^{d-1}$ was recently proved in
\cite{Wa-2011}.

  \begin{proof}
 Using monotonicity properties of the $K$-functional, it is enough  to verify inequality~\eqref{thU1_01} for
  $t=1/n,$ $n\in \mathbb{N}$.
We have 
    \begin{equation}\label{est0}
  K^{r}(f,\Opp^{(\alpha,\beta)}_r, 1/n)_{\Lw{q}{a,b}} \le C \left(\|f-P_{n,f}\|_{\Lw{q}{a,b}}+
  n^{-r}\|\Opp^{(\alpha,\beta)}_{r} P_{n,f}\|_{\Lw{q}{a,b}}\right),
  \end{equation}
  where $P_{n,f}$ is given by~\eqref{Pnf}.
  To estimate the first term, we apply~\cite[Theorem 4.1, (4.6)']{Di-Ti-2005} to  get
  \begin{equation*}
  \|f-P_{n,f}\|_{\Lw{q}{a,b}}\le C  \left(\sum_{k=n}^\infty k^{q\sigma-1}
 \|f-P_{k,f}\|_{\Lw{p}{a,b}}^{q} \right)^{1/q}.
 \end{equation*}
  In view of the realization result~\eqref{K-func2},  we obtain
  \begin{equation*}
  \begin{aligned}
  \|f-P_{n,f}\|_{\Lw{q}{a,b}}
  &\le C  \left(\sum_{k=n}^\infty k^{q\sigma-1} \|f-P_{k,f}\|_{\Lw{p}{a,b}}^{q} \right)^{1/q}
  \\
  &\le C\left(\sum_{k=n}^\infty k^{q\sigma-1}
 K^{r+\sigma}(f,\Opp^{(\alpha,\beta)}_{r+\sigma}, 1/k)_{\Lw{p}{a,b}}^{q} \right)^{1/q}
  \\
  &\le C
     \left(
    \int_0^t \left(u^{-\sigma} K^{r+\sigma}(f, \Opp^{(\alpha,\beta)}_{r+\sigma},u)_{\Lw{p}{a,b}} \right)^{q}
    \frac{du}{u}
    \right)^{1/q}.
  \end{aligned}
 \end{equation*}
 To estimate the second term in~\eqref{est0},
 we use  Theorem~D or Theorem~\ref{thHI} depending on whether $(\alpha,\beta)=(a,b)$
 or $(\alpha,\beta)\neq(a,b)$:
  \begin{align*}\label{f13}
    n^{-r}\left\|\Opp^{(\alpha,\beta)}_r P_{n,f}\right\|_{\Lw{q}{a,b}}\le
    C n^\sigma n^{-(r+\sigma)}\left\| \Opp^{(\alpha,\beta)}_{r+\sigma} P_{n,f} \right\|_{\Lw{p}{a,b}} \le
   C  n^\sigma   K^{r+\sigma}(f, \Opp^{(\alpha,\beta)}_{r+\sigma}, 1/n)_{\Lw{p}{a,b}}.
  \end{align*}
 To complete the proof of \eqref{thU1_01},
  we have 
  $$
  n^\sigma   K^{r+\sigma}(f, \Opp^{(\alpha,\beta)}_{r+\sigma}, 1/n)_{\Lw{p}{a,b}}\le C
 \left(
    \int_{1/2n}^{1/n} \left(u^{-\sigma} K^{r+\sigma}(f, \Opp^{(\alpha,\beta)}_{r+\sigma},u)_{\Lw{p}{a,b}} \right)^{q}
    \frac{du}{u}
    \right)^{1/q}.
   $$
   \end{proof}

 \section{Ulyanov-type inequalities for Ditzian--Totik moduli of smoothness}

The (global) weighted modulus of smoothness of order
{$r\ge 1$}
 is given by
  \begin{align*}
  \omega^r_\varphi(f,t)_{\Lw{p}{a,b}}=
 \Omega^r_\varphi(f,t)_{\Lw{p}{a,b}}&+
  \inf_{P\in\PP_{r-1}}\big\|(f-P)w\big\|_{\Ls{p}{[-1,\,-1+4k^2t^2]}}
 \\&+
  \inf_{P\in\PP_{r-1}}\|(f-P)w\|_{\Ls{p}{[1-4k^2t^2,\,1]}},
  \end{align*}
  where $w=\left(\w{a}{b}\right)^{1/p},$
  $$
   \Omega^r_\varphi(f,t)_{\Lw{p}{a,b}}=
  \sup_{0<h\le t}\|\Delta^r_{h\varphi}fw\|_{\Ls{p}{[-1+4k^2t^2,\,1-4k^2t^2]}}
  $$
  and
  $$
  \Delta^r_{h\varphi}f(x)=\sum_{i=0}^r(-1)^i \binom{r}{i}
  f\left(x+\tfrac{r-2i}{2}h\varphi(x)\right).
  $$
Note that  (see \cite[(2.5.7)]{Ma-Mi}) this definition
is equivalent to the one given in \cite[Chapter 6, Appendix B]{book}.

  Let  $K^r_{\varphi}(f,t)_{\Lw{p}{a,b}}$, ${r\in \mathbb{N}}$,
  be the   $K$-functional
for the pair of spaces $\left(\LL{p}{a,b}, W^r_{p, (a,b)}\right)$, where
$W^r_{p, (a,b)}$ consists of functions $g\in \LL{p}{a,b}$ such that
  $g^{(r-1)}\in \mathrm{AC}_{\mathrm{loc}}$ and
$\varphi^r g^{(r)} \in \LL{p}{a,b}$ (see \cite[(6.1.1)]{book}):
   \begin{equation}\label{K-varphi-func}
   K^r_{\varphi}(f,t)_{\Lw{p}{a,b}}=
  \inf \Big\{\|f-g\| _{\Lw{p}{a,b}}+t^r \|\varphi^r g^{(r)}\|_{\Lw{p}{a,b}}
   : \ g\in W^r_{p, (a,b)}\Big\}.
   \end{equation}
   It is  known that
   $ K^r_{\varphi}(f,t)_{p,(a,b)}\asymp  \omega^r_\varphi(f,t)_{\Lw{p}{a,b}}$
      for $a, b\ge 0$; see  
     \cite[Theorem 6.1.1]{book}. 
  Moreover, we have the following realization result:
  \begin{equation}\label{realization_classic}
  \omega^r_\varphi(f,t)_{\Lw{p}{a,b}}\asymp
  \|f-P_{n,f}\| _{\Lw{p}{a,b}}+
   t^r\|\varphi^r P_{n,f}^{(r)}\|_{\Lw{p}{a,b}}, \qquad [1/t]=n.
  \end{equation}
 The proof of this equivalence  (cf. \cite{Di-Hr-Iv-1995}) is based on the
   Jackson-type inequality and the estimate of $t^r\|\varphi^r \psi^{(r)}\|_{\Lw{p}{a,b}}$ via
$ \omega^r_\varphi(f,t)_{\Lw{p}{a,b}}$ (the Nikolskii--Stechkin type inequality).
  The Jackson-type inequality was obtained in \cite[Theorem 7.2.1]{book} for the unweighted case and in
   \cite[Sec. 2.5.2, (2.5.17)]{Ma-Mi} for the weighted case.
The unweighted version of the Nikolskii--Stechkin type inequality   
 was proved in  \cite[Theorem 7.3.1]{book}.
 This argument can be used to show the  weighted version.

  The relation between $K$-functionals \eqref{K-varphi-func} and \eqref{K-func1}
   in the case when $r$ is positive integer follows from Corollary \ref{co1} below.
Note that the case $(\alpha,\beta)=(a,b)$ is due to Dai and Ditzian \cite[Theorem~7.1]{Da-Di-2005}
and is based on the Muckenhoupt transplantation theorem.
We follow the idea of their proof and first obtain the following result.

  \begin{theorem}\label{th3}
 {\it
Let $1<p<\infty$, $r$ be a positive integer,
 and $a,\,b,\,\alpha,\,\beta >-1$ be such that $\eqref{abab}$ holds.
Then there exists a constant $C$ such that for any
  $Q\in \PP,$
  we have
    \begin{equation}\label{th3_01}
    \left\|\varphi^r Q^{(r)}\right\|_{\Lw{p}{a,b}}
    \le C \left\|\Opp^{(\alpha,\beta)}_r Q  \right\|_{\Lw{p}{a,b}},
   \end{equation}
  \begin{equation}\label{th3_02}
   \left\|\Opp^{(\alpha,\beta)}_r \left( Q-   S^{(\alpha,\beta)}_{r-1}Q\right) \right\|_{\Lw{p}{a,b}} \le
   C\left\|\varphi^r Q^{(r)}\right\|_{\Lw{p}{a,b}},
   \end{equation}
where $   S^{(\alpha,\beta)}_{r-1}Q$ is the $(r-1)$-th partial sum of the Fourier--Jacobi expansion of $Q$, i.e.,
 $$
   S^{(\alpha,\beta)}_{r-1}Q=\sum_{k=0}^{r-1}   \widehat{Q}^{(\alpha,\beta)}_k  \psi_k^{(\alpha,\beta)}.
 $$
 }

\end{theorem}

\begin{proof}
  The proof of \eqref{th3_01} and \eqref{th3_02} is based on Theorem~\ref{ThMu}.
Since ${\widehat{Q}}_k^{(\alpha,\beta)}=0$ starting from certain $k$,
   we obtain
 $$
  \Opp^{(\alpha,\beta)}_r Q= \sum_{k=1}^\infty \left(\lambda_k^{(\alpha,\beta)}\right)^{r}
 {\widehat{Q}}_k^{(\alpha,\beta)} \psi_k^{(\alpha,\beta)}=
   \sum_{k=1-r}^\infty \left(\lambda_{k+r}^{(\alpha,\beta)}\right)^{r}
 {\widehat{Q}}_{k+r}^{(\alpha,\beta)} \psi_{k+r}^{(\alpha,\beta)},
  $$
 $$
 Q^{(r)}= \sum_{k=r}^\infty \lambda_k{\widehat{Q}}_k^{(\alpha,\beta)}
 \psi_{k-r}^{(\alpha+r,\beta+r)}=
  \sum_{k=0}^\infty \lambda_{k+r}{\widehat{Q}}_{k+r}^{(\alpha,\beta)}
 \psi_{k}^{(\alpha+r,\beta+r)},
$$
where
$$
 \lambda_k=\lambda_k(\alpha,\beta,r)=
  \lambda_k^{(\alpha,\beta)}\cdots \lambda_{k-r+1}^{(\alpha+r-1,\beta+r-1)}.
  $$
 To prove  inequality~\eqref{th3_01},
  we apply Theorem~\ref{ThMu} with $(\overline{p},\overline{q})=(p,p)$,
   $(\overline{\alpha},\overline{\beta})=(\alpha+r,\beta+r)$,
  $(\overline{\gamma},\overline{\delta})=(\alpha,\beta)$,
   $(\overline{c},\overline{d})=(a,b)$,
  $h=-r$,
and   $$
    \cm{k}= \lambda_k/ \left( \lambda_k^{(\alpha,\beta)}\right)^{r}.
   $$
Then  $s=0$,  $\left(\overline{a},\overline{b}\right)=(a+pr/2, b+pr/2)$,
 $ A=(a+1)/p-\alpha,$ and $B=(b+1)/p-\beta$.
   On account of \eqref{abab}, we conclude that $A<1$, $B<1$, and  therefore,
   all conditions of Theorem~\ref{ThMu} are satisfied. Hence, we get
 $$
  \left\|\varphi^r Q^{(r)}\right\|_{\Lw{p}{a,b}} =
  \left\| Q^{(r)}\right\|_{\Lw{p}{a+pr/2,b+pr/2}}
  \le C \left\|\Opp^{(\alpha,\beta)}_r Q  \right\|_{\Lw{p}{a,b}}.
 $$

Let us now obtain~\eqref{th3_02}. We remark that
  $g=\Opp^{(\alpha,\beta)}_r \left(Q -S^{(\alpha,\beta)}_{r-1}Q\right)$
  is a polynomial and its Fourier--Jacobi  coefficients  satisfy
   $\widehat{g}^{(\alpha,\beta)}_k=0$ for $0\le k\le r-1$.
   We apply Theorem~\ref{ThMu} with $(\overline{p},\overline{q})=(p,p)$,
   $(\overline{\alpha},\overline{\beta})=(\alpha,\beta)$,
  $(\overline{\gamma},\overline{\delta})=(\alpha+r,\beta+r)$,
   $(\overline{c},\overline{d})=(a+pr/2,b+pr/2)$,
  $h=r$, and
   $$
    \cm{k}= \left( \lambda_k^{(\alpha,\beta)}\right)^{r}/\lambda_k.
   $$
   Then
    $s=0$,  $\left(\overline{a},\overline{b}\right)=(a, b)$,
 $ A=(a+1)/p-\alpha-r/2<1,$ and $B=(b+1)/p-\beta-r/2<1.$
 Therefore, all conditions of Theorem~\ref{ThMu}
are satisfied, and we arrive at
 $$
 \left\|\Opp^{(\alpha,\beta)}_r  \left(Q -S^{(\alpha,\beta)}_{r-1}Q \right)  \right\|_{\Lw{p}{a,b}}
 \le C
 \left\| Q^{(r)}\right\|_{\Lw{p}{a+pr/2,b+pr/2}}=
  C\left\|\varphi^r Q^{(r)}\right\|_{\Lw{p}{a,b}} .
 $$
 \end{proof}

  \begin{corollary}\label{co1}
 {\it
Under assumptions of Theorem $\ref{th3},$
there exists a constant $C$ such that for any
  $f\in \LL{p}{a,b}$ and $t\in(0,t_0)$ we have
 \begin{equation}\label{co1_02}
   K^r_{\varphi}(f,t)_{\Lw{p}{a,b}} \le C K^r(f, \Opp^{(\alpha,\beta)}_r,t)_{\Lw{p}{a,b}}
 \end{equation}
 and
 \begin{equation*}
 K^r(f, \Opp^{(\alpha,\beta)}_r,t)_{\Lw{p}{a,b}}\le
 C\left(   K^r_{\varphi}(f,t)_{\Lw{p}{a,b}}+t^r\|f\|_{\Lw{p}{a,b}}\right).
\end{equation*}}
\end{corollary}

 \begin{proof}
First, (\ref{th3_01}) and the realization result (\ref{realization_classic})  yield that
 \begin{align*}
   K^r_{\varphi}(f,t)_{\Lw{p}{a,b}}& \le   \|f-P_{n,f}\|_{\Lw{p}{a,b}}+
 t^{r}\|\varphi^r P_{n,f}^{(r)}\|_{\Lw{p}{a,b}}
 \\
 &\le
  C\left(  \|f-P_{n,f}\|_{\Lw{p}{a,b}}+
 t^{r}\|\Opp^{(\alpha,\beta)}_r  P_{n,f}\|_{\Lw{p}{a,b}}    \right) \le C  K^r(f, \Opp^{(\alpha,\beta)}_r,t)_{\Lw{p}{a,b}},
 \end{align*}
 which is \eqref{co1_02}.

Second, under condition~\eqref{abab}, the operator
$A: \PP\to \PP_{r-1}$ given by
 $$A({Q})=\Opp^{(\alpha,\beta)}_r S^{(\alpha,\beta)}_{r-1}{Q}$$
is bounded in $\LL{p}{a,b},$ i.e.,
 \begin{equation}\label{vvsspp}
 \|\Opp^{(\alpha,\beta)}_r  S^{(\alpha,\beta)}_{r-1}{Q}\|_{\Lw{p}{a,b}}\le C(p,a,b,\alpha,\beta,r)
\|{Q}\|_{\Lw{p}{a,b}}.
 \end{equation}
Using this, we obtain
 \begin{align*}
 &K^r(f, \Opp^{(\alpha,\beta)}_r,t)_{\Lw{p}{a,b}}
\le
 \|f-P_{n,f}\|_{\Lw{p}{a,b}}+ t^{r}\|\Opp^{(\alpha,\beta)}_r  P_{n,f}\|_{\Lw{p}{a,b}}
\\
&\le
 \|f-P_{n,f}\|_{\Lw{p}{a,b}}+
 t^{r}\|\Opp^{(\alpha,\beta)}_r ( P_{n,f} -S^{(\alpha,\beta)}_{r-1}P_{n,f})\|_{\Lw{p}{a,b}}
+ t^{r}\|\Opp^{(\alpha,\beta)}_r S^{(\alpha,\beta)}_{r-1}P_{n,f}\|_{\Lw{p}{a,b}}.
\end{align*}
Finally,
 (\ref{th3_02}) and
(\ref{vvsspp}) imply
 \begin{eqnarray*}
 K^r(f, \Opp^{(\alpha,\beta)}_r,t)_{\Lw{p}{a,b}}
&\le& C\left( \|f-P_{n,f}\|_{\Lw{p}{a,b}}+ t^{-r}\|\varphi^r P_{n,f}^{(r)} \|_{\Lw{p}{a,b}} +
   t^{r}\| P_{n,f}\|_{\Lw{p}{a,b}}\right)
\\
&\le&
 C\,\big(   K^r_{\varphi}(f,t)_{\Lw{p}{a,b}}+t^r\|f\|_{\Lw{p}{a,b}}\big).
\end{eqnarray*}

  \end{proof}


It is proved in \cite[Theorem 11.2]{Di-Ti-2005} that for $f\in L_p,$ $0<p<q\le \infty$, and integer $r\ge 1$
the following Ulyanov-type inequality holds:
$$
   \omega^r_\varphi\left(f, t\right)_{\Ll{q}} \le C
      \left[
    \int_0^{t}\left(u^{-\sigma}\omega^{r}_\varphi(f,u)_{\Ll{p}} \right)^{q_1} \frac{du}{u}
    \right]^{1/q_1},
$$
where
 $q_1=\begin{cases}
        q, & q<\infty\\
        1, & q=\infty\end{cases},
  $ \  $\sigma=2\left(\frac1p-\frac1q \right).$
The next theorem refines this result.

\begin{theorem} \label{thU}
{\it
Let  $1\le p<q\le \infty,$ $a \ge b\ge 0,$  $r$  be a positive  integer$,$ and
 $$
\sigma =(2a+2)\left(\frac1p-\frac1q\right).
 $$
 Suppose that
  $f\in \LL{p}{a,b}$ and
$$    \int_0^1 \left(u^{-\sigma}\omega^{r+[\sigma]}_\varphi(f,u)_{\Lw{p}{a,b}} \right)^{q_1} \frac{du}{u}<\infty.
$$
 Then $f\in \LL{q}{a,b}$
and
  \begin{equation}\label{thU01}
   \omega^r_\varphi\left(f, t\right)_{\Lw{q}{a,b}} \le C
      \left[
    \int_0^{t}\left(u^{-\sigma}\omega^{r+[\sigma]}_\varphi(f,u)_{\Lw{p}{a,b}} \right)^{q_1} \frac{du}{u}
    \right]^{1/q_1} +
    Ct^r E_{r-1}(f)_{\Lw{p}{a,b}},
   \end{equation}
   where
   $$q_1=\begin{cases}
        q, & q<\infty,\\
        1, & q=\infty.
        \end{cases}
   $$

}  \end{theorem}
 \begin{remark}
(A).\ \ \  In particular, (\ref{thU01}) implies
$$   \omega^r_\varphi\left(f, t\right)_{q} \le C
      \left[
    \int_0^{t}\left(u^{-1}\omega^{r+1}_\varphi(f,u)_{p} \right)^{q_1} \frac{du}{u}
    \right]^{1/q_1} +
    Ct^r E_{r-1}(f)_{p},
$$
when $\frac1p-\frac1q\ge \frac 12$, $1\le p<q\le \infty$, 
and
$$   \omega^r_\varphi\left(f, t\right)_{\infty} 
\le C
    \int_0^{t} u^{-2}\omega^{r+2}_\varphi(f,u)_{1}  \frac{du}{u}
     +
    Ct^r E_{r-1}(f)_{1}.
$$
(B).\ \ \ 
Corollary~\ref{co1} shows that for $1<p<q<\infty$ and
   positive integer $\sigma$
 Theorem~\ref{thU} follows from Theorem~\ref{thU1}.
\end{remark}

\begin{proof}
  The proof is similar to the proof of Theorem~\ref{thU1}.
  The only substantial  difference is that we use Lemma~\ref{le1}   instead of Theorem~D and Theorem~\ref{thHI}.

 Using monotonicity properties of the moduli of smoothness, it is enough  to verify inequality~\eqref{thU01} for
  $t=1/n,$ where $n$ is a positive integer.
  Let $P_{n,f}$ be defined by~\eqref{Pnf}.
   Taking into account that
   $\omega^r_\varphi(f,t)_{\Lw{{q}}{a,b}} \asymp K^r_{\varphi}(f,t)_{\Lw{{q}}{a,b}}$, 
   we obtain
    \begin{equation}\label{thU02}
  \omega^r_\varphi\left(f, t\right)_{\Lw{q}{a,b}}  \le C \left(\|f-P_{n,f}\|_{\Lw{q}{a,b}}+
  n^{-r}\|\varphi^r P_{n,f}^{(r)}\|_{\Lw{q}{a,b}}\right).
  \end{equation}

  To estimate the first term, we apply Theorem 4.1 from \cite{Di-Ti-2005}.
  Assumption (4.3) of this theorem is exactly the Nikol'skii inequality
   \begin{equation*}
   \|P_n\|_{\Lw{q}{a,b}}\le Cn^{(2a+2)\left(\frac1p-\frac1q\right)}\|P_n\|_{\Lw{p}{a,b}}, \quad P_n\in \Pi_n,
   \end{equation*}
   where   $C=C(p,q,a,b)$,  proved in~\cite[Theorem 4]{Da-Ra-1972} (see also~\cite[Ch.~6, Theorem~1.8.4, 1.8.5]{Mi-1994}).
   Therefore, we have
  \begin{equation*}
  \|f-P_{n,f}\|_{\Lw{q}{a,b}}\le C  \left(\sum_{k=n}^\infty k^{q_1\sigma-1}
 \|f-P_{k,f}\|_{\Lw{p}{a,b}}^{q_1} \right)^{1/q_1}.
 \end{equation*}
  Applying~\eqref{realization_classic} and replacing  the sum by the integral,  we get
  \begin{equation*}
  \begin{aligned}
  \|f-P_{n,f}\|_{\Lw{q}{a,b}}&\le C  \left(\sum_{k=n}^\infty k^{q_1\sigma-1}
 \|f-P_{k,f}\|_{\Lw{p}{a,b}}^{q_1} \right)^{1/q_1}
   \\
  &\le C\left(\sum_{k=n}^\infty k^{q_1\sigma-1}
 \omega^{r+[\sigma]}_\varphi(f, 1/k)_{\Lw{p}{a,b}}^{q_1} \right)^{1/q_1}
   \\
   &\le C
   \left( \int_0^t\left(u^{-\sigma}\omega^{r+[\sigma]}_\varphi(f,u)_{\Lw{p}{a,b}} \right)^{q_1} \frac{du}{u}\right)^{1/q_1}.
  \end{aligned}
 \end{equation*}

 To estimate the second term  in~\eqref{thU02},
 we use Lemma~\ref{le1}:
 \begin{align*}
  \left\|\varphi^rP_n^{(r)}\right\|_{\Lw{q}{a,b}}&=
   \left\|\varphi^r(P_n- P_{r-1})^{(r)}\right\|_{\Lw{q}{a,b}}\le \left\|P_n-P_{r-1}\right\|_{\Lw{p}{a,b}}+
  \left\|\varphi^{r+2[\sigma]-\sigma}P_n^{(r+[\sigma])}\right\|_{\Lw{p}{a,b}}.
 \end{align*}
Further we need
 the following two-weight inequality proved in \cite[Theorem 4]{Da-Ra-1972}:
\begin{align*}
 \left\|\varphi^{r+2[\sigma]-\sigma}P_n^{(r+[\sigma])}\right\|_{\Lw{p}{a,b}}
  \le   C \,n^{\sigma-[\sigma]} \left\|\varphi^{r+[\sigma]}P_n^{(r+[\sigma])}\right\|_{\Lw{p}{a,b}}.
  \end{align*}
  Therefore, using monotonicity properties of moduli of smoothness, we get
\begin{eqnarray*}
 n^{-r} \left\|\varphi^{r+2[\sigma]-\sigma}P_n^{(r+[\sigma])}\right\|_{\Lw{p}{a,b}}
  &\le& C \,n^\sigma  \omega^{r+[\sigma]}_\varphi(f, 1/n)_{\Lw{p}{a,b}}
  \\
  &\le& C
\left[ \int_{1/2n}^{1/n}\left(u^{-\sigma}\omega^{r+[\sigma]}_\varphi(f,u)_{\Lw{p}{a,b}} \right)^{q_1} \frac{du}{u}
   \right]^{1/q_1}.
  \end{eqnarray*}
 To complete the proof  we note that $\|P_n-P_{r-1}\|_{\Lw{p}{a,b}}\le 2E_{r-1}(f)_{\Lw{p}{a,b}}$.
 \end{proof}

{\bf Acknowledgement.}
  The authors would like to thank
   F. Dai, Z. Ditzian, and  G. Mastroianni for fruitful discussions and useful comments on the fractional $K$-functionals,
  and the  referee for reading the paper carefully and several  valuable   comments.


 \end{document}